\documentclass[a4paper]{article}
\usepackage{fullpage}
\usepackage{xcolor}
\usepackage{amssymb}
\usepackage{amsmath, amsthm, amsfonts}
\usepackage{float}
\usepackage{graphicx}
\usepackage{lineno}
\usepackage{enumitem}
\usepackage{longtable}
\usepackage{environ}
\usepackage{bm}
\usepackage{thm-restate}
\usepackage{hyperref}
\usepackage{cleveref,lipsum}

\usepackage[numbers]{natbib}
\usepackage{xurl}

\usepackage{soul}

\hypersetup{
	colorlinks=true, 
	linkcolor=cyan,
	citecolor=blue,
}

\usepackage{graphicx,accents}

\makeatletter
\DeclareRobustCommand{\ovl}[1]{%
	\mathpalette\do@cev{#1}%
}
\newcommand{\do@cev}[2]{%
	\fix@cev{#1}{+}%
	\reflectbox{$\m@th#1\ovr{\reflectbox{$\fix@cev{#1}{-}\m@th#1#2\fix@cev{#1}{+}$}}$}%
	\fix@cev{#1}{-}%
}
\newcommand{\fix@cev}[2]{%
	\ifx#1\displaystyle
	\mkern#23mu
	\else
	\ifx#1\textstyle
	\mkern#23mu
	\else
	\ifx#1\scriptstyle
	\mkern#22mu
	\else
	\mkern#22mu
	\fi
	\fi
	\fi
}

\usepackage{mathtools}  
\usepackage{pgfplots}

\usepackage{amsmath,amssymb,amsfonts}
\usepackage{mathtools}
\usepackage{chngcntr}

\newtheorem{theorem}{Theorem}[section]

\newtheorem{corollary}[theorem]{Corollary}

\newtheorem{conjecture}[theorem]{Conjecture}

\newtheorem{proposition}[theorem]{Proposition}

\newtheorem{question}[theorem]{Question}
\newtheorem*{thm:function}{Theorem~\ref{thm:function}} 

\usepackage{todonotes}

\date{}

\newcommand{\len}{\operatorname{length}}

\title{Two-block paths in oriented graphs of large semidegree} 

\author{Irena Penev\thanks{Computer Science Institute (I\'UUK, MFF), Charles University, Prague, Czech Republic. Supported by GA\v{C}R grant 25-17377S. \texttt{ipenev@iuuk.mff.cuni.cz}. } ~~ S Taruni\thanks{Centro de Modelamiento Matemático (CNRS IRL2807), Universidad de Chile, Santiago, Chile. Supported by Centro de Modelamiento Matemático (CMM) BASAL fund FB210005 for center of excellence from ANID-Chile and by MSCA-RISE-2020-101007705 project \textit{RandNET}. \texttt{tsridhar@cmm.uchile.cl}. } ~~ St\'ephan
Thomass\'e\thanks{Univ. Lyon, ENS de Lyon, UCBL, CNRS, LIP, France. \texttt{stephan.thomasse@ens-lyon.fr}.} ~~
Ana Trujillo-Negrete\thanks{Facultad de Ciencias, Universidad Nacional Autónoma de México, Mexico City, Mexico.  Supported by ANID/Fondecyt Postdoctorado 3220838, by ANID Basal Grant CMM FB210005, by MSCA-RISE-2020-101007705 project \textit{RandNET}, and by the Universidad Nacional Autónoma de México Postdoctoral Program (POSDOC). \texttt{ltrujillo@ciencias.unam.mx}. } ~~ Mykhaylo Tyomkyn\thanks{Department of Applied Mathematics (KAM, MFF), Charles University, Prague, Czech Republic. Supported by GA\v{C}R grant 25-17377S and ERC Synergy Grant DYNASNET 810115. \texttt{tyomkyn@kam.mff.cuni.cz}.}}

\begin{document}
	\maketitle

    \begin{abstract}
        We study the existence of oriented paths with two blocks in oriented graphs under semidegree conditions. A \textit{block} of an oriented path is a maximal directed subpath. Given positive integers  $k$  and  $\ell$  with  $k/2\le \ell < k$, we establish a semidegree function that guarantees the containment of every oriented path with two blocks of sizes  $\ell$  and  $k-\ell$. As a corollary, we show that every oriented graph with all in- and out-degrees at least $3k/4$ contains every two-block path with $k$ arcs. Our results extend previous work on Stein's conjecture and related problems concerning oriented paths. \\
        
    \end{abstract}

\section{Introduction}

The problem of determining the existence of long paths in a graph with minimum degree constraints is well-studied~\cite{dirac1952some,posa1976hamiltonian}.
In particular, Dirac~\cite{dirac1952some} proved that a minimum degree of at least ${n}/{2}$ guarantees the existence of a Hamilton cycle in a graph with $n$ vertices, while a minimum degree of at least ${(n-1)}/{2}$ ensures the containment of a Hamilton path. This problem naturally extends to oriented graphs and digraphs. An \textit{oriented graph} is a digraph that contains no directed cycles of length one or two. As a natural analogue to the minimum degree in graphs, the concept of minimum semidegree has been used in the study of oriented Hamilton cycles~\cite{ghouilahouri1960condition,haggkvist1993hamilton,haggkvist1995oriented,keevash2009exact,kelly2009arbitrary,kelly2008dirac}. The \textit{minimum semidegree} of a digraph $D$, denoted by $\delta^{0}(D)$, is defined as the minimum value among the in-degrees and out-degrees of all vertices in $D$.

A related problem involves finding paths of a given size in a graph. Erd\H{o}s and Gallai~\cite{gallai1959maximal} proved that a connected graph with at least $k+1$ vertices and minimum degree at least $k/2$ contains a $k$-edge path.
In oriented graphs Jackson~\cite{jackson} proved that $\delta^{0}(G) \ge k/2$ implies $G$ contains a directed path, that is a path where all arcs have the same direction, with $k$ arcs.
More recently, Stein~\cite{stein_conjecture} conjectured a generalization of Jackson’s theorem to all oriented paths. 

\begin{conjecture}[Stein,~\cite{stein_conjecture}]\label{conj:stein}
Every oriented graph $G$ with $\delta^0(G)>k/2$ contains every orientation of the $k$-edge path. 
\end{conjecture}

Note that Conjecture~\ref{conj:stein} becomes straightforward when the minimum semidegree bound is strengthened to $\delta^0(G) \ge k$, as a simple greedy embedding succeeds in all cases. However, the problem becomes significantly more challenging when working with weaker bounds. Thus, in the pursuit of proving Stein's conjecture, the study of certain orientations of the path offers a direction for further exploration. As noted above, Jackson's result~\cite{jackson} confirms Conjecture~\ref{conj:stein} for the case of directed paths. Moreover, Stein and Trujillo-Negrete~\cite{stein-trees} confirmed Conjecture~\ref{conj:stein}, as a consequence of a more general result, for the class of oriented graphs containing no oriented $4$-cycle. 

Further progress has been achieved for antidirected paths. An \textit{antidirected path} is an oriented path that alternates the direction of its arcs. Klimo\v{s}ová and Stein~\cite{klimosova} showed that every oriented graph $G$ with minimum semidegree at least $3k/4$ contains each antidirected path with $k$ arcs. Their work was later improved by Chen, Hou, and Zhou~\cite{chen}, by reducing the required semidegree to $2k/3$, and further to $5k/8$ by Skokan and Tyomkyn~\cite{skokan}. 
Most recently, Grzesik and Skrzypczyk~\cite{grzesik} proved that 
$\delta^0(G)> \tfrac12\bigl(k-1+\sqrt{k-3}\bigr)$ suffices for every antidirected path of length $k$, matching the conjectured $k/2$ threshold up to an $O(\sqrt{k})$ term.

The next most natural class of oriented paths to consider are perhaps the paths with two blocks. A \textit{block} of an oriented path $P$ is a maximal directed subpath of $P$, and its \textit{size} is the number of its arcs. The containment of paths with two blocks in oriented graphs and digraphs has been previously studied under chromatic number constraints~\cite{chromaticconjecture,el2004paths,sahili2007paths}. El-Sahili~\cite{el2004paths} conjectured that a digraph with chromatic number $k$ contains as a subgraph every $k$-arc oriented path with two blocks. This conjecture was confirmed by Addario-Berry, Havet and Thomass\'e~\cite{chromaticconjecture}. 

In this paper, we study the existence of paths with two blocks in oriented graphs under minimum semidegree conditions. Our main result is the following.

\begin{restatable}{theorem}{function}
	\label{thm:function}
        \label{sec:function}
	Let $k$ and $\ell$ be positive integers with ${k}/{2}\le \ell<k$. Let $G$ be an oriented graph such that 
	\[\delta^0(G)\ge \begin{cases}
		k-\frac{\ell}{2} & \textrm{if $\ell\le \frac{2k}{3}$}, \\
		\frac{2k}{3} & \textrm{if $\ell> \frac{2k}{3}$.}
	\end{cases}\]
	Then $G$ contains as a subgraph both possible $k$-arc oriented paths with two blocks of size $\ell$ and $k-\ell$.   
\end{restatable}

The semidegree function in Theorem~\ref{thm:function} is depicted in Figure~\ref{fig:function}. 
As a direct consequence, we obtain a general bound on $\delta^0(G)$ for the containment of paths with two blocks, irrespective of the size of the blocks. 

\begin{corollary}
    Let $k \ge 2$ be an integer and let $G$ be an oriented graph with $\delta^0(G) \ge {3k}/{4}$.  
	Then $G$ contains as a subgraph every $k$-arc oriented path with two blocks.    
\end{corollary}

\begin{figure}[t]
	\centering
\begin{tikzpicture}
	\begin{axis}[
		axis lines=middle,
		xlabel style={at={(axis description cs:1.1,-0.05)}},
		xlabel={$\ell$},
		ylabel={$\delta^0(G)$},
		ylabel style={at={(axis description cs:-0.1,1.13)}},
		xtick={0.5, 0.66, 1},
		xticklabels={$\frac{k}{2}$, $\frac{2k}{3}$, $k$},
		ytick={0.5, 0.66, 0.75},
		yticklabels={$\frac{k}{2}$, $\frac{2k}{3}$, $\frac{3k}{4}$},
		ymin=0.5, ymax=0.8,
		xmin=0.5, xmax=1,
		enlargelimits=true,
		grid=major
		]
		
		\addplot[blue, thick] coordinates {
			 (0.5,0.75) (0.66,0.66)
		};

		\addplot[blue, thick] coordinates {
			(0.66,0.66) (1,0.66)
		};
		
	\end{axis}
\end{tikzpicture}
\caption{The semidegree function given in Theorem~\ref{thm:function}.}
\label{fig:function}
\end{figure}

\section{Notation}\label{sec:not}
For an oriented graph $G$, $V(G)$ denotes the set of vertices and $A(G)$ the set of arcs in $G$.  Given a vertex $v$ of an oriented graph $G$, let $N^+(v) = \{u \in V(G) : (v,u) \in A(G) \}$ and $N^-(v) = \{u \in V(G): (u,v) \in A(G)\}$ denote its \textit{out-neighbourhood} and \textit{in-neighbourhood}, respectively. The \textit{out-degree} and \textit{in-degree} of $v$ are 
$\deg^+(v)=|N^+(v)|$ and $\deg^-(v)=|N^-(v)|$, respectively. The \textit{minimum out-degree}
and \textit{minimum in-degree} of $G$ are 
\[\delta^+(G):=\min_{v\in V(G)}\{\deg^+(v)\} \quad \textrm{and}\quad  \delta^-(G):=\min_{v\in V(G)}\{\deg^-(v)\},\]
respectively. 
Additionally, we define the \textit{minimum semidegree} of $G$ as $\delta^0(G):=\min \{\delta^+(G),\delta^-(G)\}$. 
For $S \subseteq V(G)$, $G[S]$ denotes the subgraph of $G$ induced by $S$. A directed cycle $C$ in $G$ is said to be a \textit{Hamilton cycle} if $V(C)=V(G)$.  

The \textit{length} of a path $P$, denoted by $\len(P)$, is defined as the number of its arcs. 
Given a directed path $P=u_1u_2\ldots u_p$, its \textit{reverse} is the path  
$\overleftarrow{P}=u_pu_{p-1}\ldots u_1$. (Here we emphasize that as oriented graphs $P$ and $\protect\overleftarrow{P}$ are exactly the same. The only difference is in the order of vertices that we use to represent the path.) For two vertices $u_i$ and $u_j$ in $P$, with $i<j$, let $u_i\xrightarrow{P}u_j$ be the subpath of $P$ given by $u_iu_{i+1}\ldots u_j$. 
Given two directed paths $P_1$ and $P_2$ such that their initial vertices coincide, we write $\protect\overleftarrow{P_1}P_2$ to denote the concatenation of $\protect\overleftarrow{P_1}$ and $P_2$. See an example in Figure~\ref{fig:concatenation}. 
We denote by $P(\overleftarrow{r},\protect\overrightarrow{s})$ the $(r+s)$-arc path with two blocks, where the first block consists of $r$ backward arcs followed by a second block of $s$ forward arcs. Similarly, the notation $P(\protect\overrightarrow{r},\protect\overleftarrow{s})$ represents the reverse configuration. See Figure~\ref{fig:two-blocks} for two examples. 

\begin{figure}[t]
	\centering
	\includegraphics[width=0.35\textwidth]{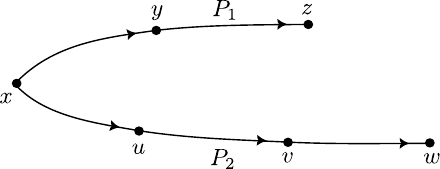}
	\caption{Concatenation $\protect\overleftarrow{P_1}P_2$, where $P_1 = xyz$ and $P_2=xuvw$. }
	\label{fig:concatenation}
\end{figure}

\begin{figure}[t]
    \centering
    \includegraphics[width=0.45\textwidth]{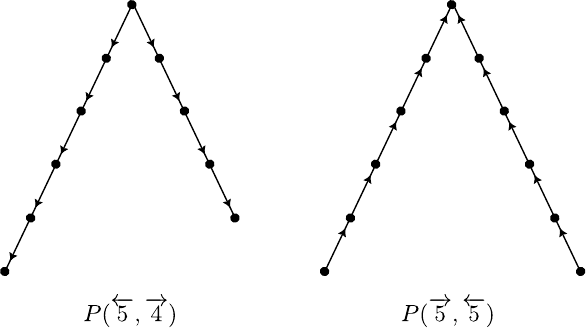}
    \caption{Examples of paths with two blocks. On the left, $P(\protect\overleftarrow{5},\protect\overrightarrow{4})$ consists of a first block of 5 backward arcs followed by a second block of 4 forward arcs. On the right, $P(\protect\overrightarrow{5},\protect\overleftarrow{5})$ consists of a first block of 5 forward arcs followed by a second block of 5 backward arcs. }
    \label{fig:two-blocks}
\end{figure}
Given a path $P$ and an oriented graph $G$, an {\it embedding} from $P$ to $G$ is an injective 
function $f: V(P) \to V(G)$ preserving adjacencies, that is, for each arc $(u,v)$ of $P$, we have that $(f(u),f(v))$ is an arc of $G$.
If such an embedding exists, then we say that $P$ \textit{embeds} in $G$. In this paper, we interchangeably say that a path $P$ embeds in $G$ or that $G$ contains the path $P$. 

\section{Preliminary results}

We first state a theorem of Jackson~\cite{jackson} on directed paths which will be a useful tool in the proof of Theorem~\ref{thm:function}.  

\begin{theorem}[Jackson~\cite{jackson}]\label{thm:jackson}
    Every oriented graph $G$ contains a directed path on $2\delta^0(G)$ arcs.  
\end{theorem}

\noindent
In the following proposition we gather some properties of $G$ that ensure an embedding of $P(\protect\overleftarrow{\ell},\protect\overrightarrow{k-\ell})$.

\begin{proposition}
	\label{prop:special_configurations}
Let $k$ and $\ell$ be integers with ${k}/{2}\le \ell < k$. Let $G$ be an oriented graph with $\delta^0(G)\ge k-\ell$. Let $P=v_0v_1\ldots v_t$ be a directed path of maximum length in $G$, and further suppose that $t\ge 2\ell$. Define the sets
\[X=\{v_0,\ldots,v_{k-\ell-1}\}, \quad Y= \{v_{k-\ell},\ldots,v_{t-k+\ell}\} \quad \textrm{and} \quad Z=\{v_{t-k+\ell+1},\ldots,v_t\}.\]
Suppose one of the following conditions holds:
\begin{enumerate}[label=(\roman*)]
	\item $N^-(v_0)\cap Y\ne \emptyset$ or $N^+(v_t)\cap Y\ne \emptyset$;
	\item $N^+(v_{i-1})\not\subseteq V(P)$ for some $v_{i}\in N^+(v_t)\cap (X\setminus \{v_0\})$; 
	\item $G[V(P)]$ contains a Hamilton cycle.
\end{enumerate}
Then $P(\protect\overleftarrow{\ell},\protect\overrightarrow{k-\ell})$ embeds in $G$. 
\end{proposition}
\begin{proof}
Observe that $X$, $Y$ and $Z$ are non-empty and pairwise disjoint. Moreover, $|X| = |Z| = k-\ell$. 
	We proceed by cases according to the proposition. 
	\begin{enumerate}[label=($\roman*)$]
		\item First, suppose that $N^-(v_0)\cap Y\ne \emptyset$, and let $v_i\in N^-(v_0)\cap Y$. We define the paths
		\[P_1:=v_iv_0\xrightarrow{P}v_{i-1} \quad \textrm{and }\quad P_2:=v_i\xrightarrow{P}v_t.\]
		See Figure~\ref{fig:empty_intersection} (left). We have $\len(P_1)=i$ and $\len(P_2)=t-i$. Moreover, since 
		$X\subsetneq V(P_1)$ and $Z\subsetneq V(P_2)$, it follows that $\len(P_1),\len(P_2)\ge k-\ell$. 
        
		Next, we show that at least one of $P_1$ and $P_2$ has length at least $\ell$. Note that this readily implies that $P(\protect\overleftarrow{\ell},\protect\overrightarrow{k-\ell})$ is a subpath 
		of $\protect\overleftarrow{P_1}P_2$, and thus embeds in $G$. 
		 Observe that if $i\ge \ell$, then $\len(P_1)=i\ge \ell$, whereas if $i< \ell$, then $\len(P_2)=t-i> t-\ell\ge \ell$ (because $t\ge 2\ell$). 
		
		Now, suppose that $N^+(v_t)\cap Y\ne \emptyset$, and let $v_j\in N^+(v_t)\cap Y$. By the maximality of $P$, we have that $N^-(v_0) \subseteq V(P) \setminus \{v_0\}$, and by the previous case, we may assume
		that $N^-(v_0)\cap Y= \emptyset$. It follows that $N^-(v_0) \subseteq (X \setminus \{v_0\}) \cup Z$. Since $\delta^0(G)\ge k-\ell=|X| > |X \setminus \{v_0\}|$, we must have $N^-(v_0)\cap Z\ne \emptyset$. Let $v_i\in N^-(v_0)\cap Z$. We define the paths
		\[P_1:=v_iv_0\xrightarrow{P}v_{j-1} \quad \textrm{and }\quad P_2:=v_i\xrightarrow{P}v_tv_j\xrightarrow{P}v_{i-1}.\]
		See Figure~\ref{fig:empty_intersection} (right). 
		We have $\len(P_1)=j$ and $\len(P_2)=t-j$. Again, as we have $X\subsetneq V(P_1)$ and 
		$Z\subsetneq V(P_2)$, it follows that $\len(P_1),\len(P_2)\ge k-\ell$. Furthermore, if $j\ge \ell$, then $\len(P_1)\ge \ell$, and if $j<\ell$, then $\len(P_2)> t-\ell\ge \ell$. 
		Thus, $P(\protect\overleftarrow{\ell},\protect\overrightarrow{k-\ell})$ is a subpath of $\protect\overleftarrow{P_1}P_2$.

		\begin{figure}[t]
		\centering
		\includegraphics[width=0.85\textwidth]{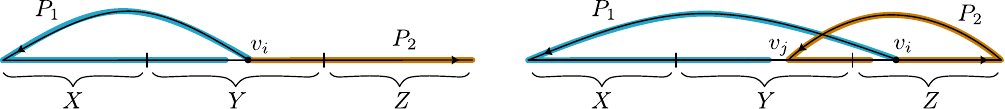}
		\caption{Construction of paths $P_1$ and $P_2$ in the proof of Proposition~\ref{prop:special_configurations}~$(i)$.}
		\label{fig:empty_intersection}
		\end{figure} 
		
		\item Fix a vertex $v_i\in N^+(v_t)\cap (X\setminus \{v_0\})$ such that $N^+(v_{i-1})\not\subseteq V(P)$. Let $w_0 \in N^+(v_{i-1}) \setminus V(P)$, and among all the directed paths in $G[V(G) \setminus V(P)]$ starting at $w_0$, choose a directed path $P' = w_0  \ldots w_m$ of maximum length.  
		Define the cycle $C := v_i \xrightarrow{P} v_t v_i$.  
		Since $v_i\in X$ and $t\ge 2\ell\ge k$, it follows that 
		\[|C| > |Y \cup Z| = |V(P)|-|X| = (t+1)-(k-\ell) \geq \ell+1.  \] 
  		Suppose first that $m\ge k-\ell-1$, and consider the following paths: 
		\[P_1:=v_{i-1}w_0\xrightarrow{P'}w_m \quad \textrm{and} \quad P_2:=v_{i-1}\xrightarrow{P}v_t.\]
		We have $\len(P_1)= m+1\ge k-\ell$ and $\len(P_2)= |C|> \ell$. Thus,  $P(\protect\overleftarrow{\ell},\protect\overrightarrow{k-\ell})$ is a subpath of
		$\protect\overleftarrow{P_1}P_2$. 
		
		Suppose now that $m \leq k-\ell-2$. By the maximality of $P'$, we have that \[N^+(w_m)\subseteq V(P)\cup (V(P') \setminus \{w_m\}).\] We distinguish two cases: 
		\begin{itemize}
		\item Suppose that $N^+(w_m)\cap V(C)\ne \emptyset$, and let $v_r\in N^+(w_m)\cap V(C)$. Define
		
        \[P'':=\begin{cases}
        v_0\xrightarrow{P}v_{i-1}w_0\xrightarrow{P'}w_mv_r\xrightarrow{P}v_tv_i\xrightarrow{P}v_{r-1} & \textrm{ if $r\neq i$},\\
        v_0\xrightarrow{P}v_{i-1}w_0\xrightarrow{P'}w_mv_i\xrightarrow{P}v_t & \textrm{ if  
        $r=i$.}
        \end{cases}
        \]
		\begin{figure}[t]
			\centering
            \includegraphics[width=0.4\textwidth]{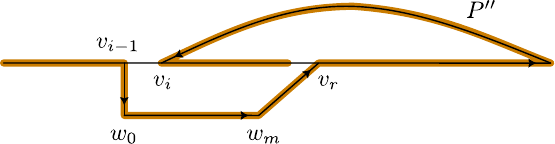}
            \caption{Construction of $P''$ in the proof of Proposition~\ref{prop:special_configurations}(ii) for the case $N^+(w_m)\cap V(C)\neq\emptyset$ (subcase $r\neq i$).}
            \label{fig:longer-path}
		\end{figure} 
		See Figure~\ref{fig:longer-path} for the subcase $r\neq i$. Since $\len(P'') = \len(P)+m+1 > \len(P)$, this contradicts the maximality of $P$.
		
		\item Suppose that $N^+(w_m)\cap V(C)= \emptyset$, so that $N^+(w_m)\subseteq (V(P)\setminus V(C))\cup (V(P') \setminus \{w_m\})$. 
		Since \[|V(P') \setminus \{w_m\}|=m\le k-\ell-2 \leq \delta^0(G)-2,\]
        it follows that $w_m$ has at least two out-neighbors in $V(P) \setminus V(C) = \{v_0,\dots,v_{i-1}\}$. Consequently, we have that $i \geq 2$, and that $w_m$ has at least one out-neighbor in $\{v_0,\dots,v_{i-2}\}$. Let $j\in \{0,\ldots,i-2\}$ be the smallest index such that $v_j\in N^+(w_m)$. We then define the paths \[P_1:=v_{i-1}w_0\xrightarrow{P'}w_mv_j\xrightarrow{P}v_{i-2} \quad \textrm{and} \quad P_2:=v_{i-1}\xrightarrow{P}v_t.\] Then $\{w_m\} \cup N^+(w_m) \subseteq V(P_1)$, and consequently, $\len(P_1) = |V(P_1)|-1 \geq |N^+(w_m)| \geq \delta^0(G) \geq k-\ell$. On the other hand, $\len(P_2) = |V(C)| > \ell$. Thus,  $P(\protect\overleftarrow{\ell},\protect\overrightarrow{k-\ell})$ embeds in $G$. 
	\end{itemize}
	
	\item Without loss of generality, suppose the Hamilton cycle is $v_0v_1\ldots v_t v_0$; if not, we may relabel the vertices in $P$. 
	Then, all the in- and out-neighbours of $v_0$ must belong to $V(P)$, for otherwise we could find a path in $G$ longer than  $P$, contradicting its maximality. Thus, $N^+(v_0) \cup N^-(v_0) \subseteq V(P) \setminus \{v_0\} = (X \setminus \{v_0\}) \cup Y \cup Z$. So, since $|N^+(v_0) \cup N^-(v_0)| \geq 2\delta^0(G) \geq 2(k-\ell) > |X \setminus \{v_0\}|+|Z|$, it follows that $(N^+(v_0)\cup N^-(v_0))\cap Y\ne \emptyset$. Let $v_i\in (N^+(v_0)\cup N^-(v_0))\cap Y$. By case $(i)$, we may assume that  $v_i\in N^+(v_0)$. Define the paths
	\[P_1:=
		v_0\xrightarrow{P}v_{i-1} 
	\quad \textrm{and} \quad 
	P_2:=
		v_0v_i\xrightarrow{P}v_{t}.	
	\]
	See Figure~\ref{fig:spanning-cycle}. 	
	
	Note that $X\subsetneq V(P_1)$ and $Z\subsetneq V(P_2)$, 
	which implies that $\len(P_1),\len(P_2)\ge k-\ell$. Moreover, at least one of $P_1$ and $P_2$ has length at least $t/2\ge \ell$. 
	Thus, $P(\protect\overleftarrow{\ell},\protect\overrightarrow{k-\ell})$  is a subpath of $\protect\overleftarrow{P_1}P_2$, and therefore embeds in $G$. 
	\end{enumerate}
\end{proof}

\begin{figure}[h]
		\centering
		\includegraphics[width=0.4\textwidth]{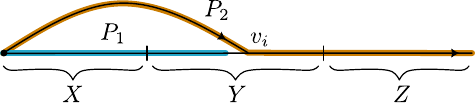}
		\caption{Construction of paths $P_1$ and $P_2$ in the proof of Proposition~\ref{prop:special_configurations}~$(iii)$.}
		\label{fig:spanning-cycle}
	\end{figure}

\section{Proof of Theorem~1}

Note that it is enough to show that $P(\protect\overleftarrow{\ell},\protect\overrightarrow{k-\ell})$
embeds in $G$, as the embedding of $P(\protect\overrightarrow{\ell},\protect\overleftarrow{k-\ell})$ within $G$ follows by embedding 
$P(\protect\overleftarrow{\ell},\protect\overrightarrow{k-\ell})$ into the oriented graph obtained by reversing each arc in $G$.

Consider a directed path $P=v_0v_1\ldots v_t$ of maximum length in $G$. 
By Theorem~\ref{thm:jackson} we have $t\ge 2\delta^0(G)$. 
We proceed by cases. 
\begin{itemize}
	\item Suppose that $\boldsymbol{k/2\le \ell \le {2k}/{3}}$. Here we have $\delta^0(G) \geq k-\ell/2$, and consequently, $t \geq 2\delta^0(G) \geq 2k-\ell$. It follows that 
	\begin{equation}
		\label{eq:diameter}
		\delta^0(G)> k-\ell \quad \textrm{and} \quad  t-\ell\ge 2k-2\ell \ge \frac{2k}{3}\ge \ell, 
	\end{equation}
which in particular implies that $t \geq 2\ell$. Define the sets 
	\[X=\{v_0,\ldots,v_{k-\ell-1}\}, \quad Y= \{v_{k-\ell},\ldots,v_{t-k+\ell}\} \quad \textrm{and}\quad  Z=\{v_{t-k+\ell+1},\ldots,v_t\}.\]
	Note that with this set-up, if any one of the conditions $(i)$, $(ii)$ and $(iii)$ from Proposition~\ref{prop:special_configurations} holds, then $P(\protect\overleftarrow{\ell},\protect\overrightarrow{k-\ell})$  embeds in $G$, and we are done. 

We have that 
	\[|X|=|Z|=k-\ell \quad \textrm{and} \quad|Y|=t+1-2(k-\ell)\ge (2k-\ell)+1-2(k-\ell) = \ell+1. \]
	By the maximality of $P$, we have $N^-(v_0), N^+(v_t)\subseteq V(P)$. 
	Note that if $N^-(v_0)\cap Y\ne \emptyset$ or $N^+(v_t)\cap Y \ne \emptyset$, then Proposition~\ref{prop:special_configurations}~$(i)$ guarantees that $P(\protect\overleftarrow{\ell},\protect\overrightarrow{k-\ell})$  embeds in $G$. Moreover, if $v_0\in N^+(v_t)$ (equivalently: $v_t \in N^-(v_0)$), then $v_0v_1 \dots v_tv_0$ is a Hamilton cycle of $G[V(P)]$, and so by Proposition~\ref{prop:special_configurations}~$(iii)$, we again obtain that  $P(\protect\overleftarrow{\ell},\protect\overrightarrow{k-\ell})$  embeds in $G$. Thus, we may assume that 
	\begin{equation}
		\label{eq:neighbourhoods}
		N^-(v_0)\subseteq  (X\setminus \{v_0\})\cup (Z \setminus \{v_t\}) \quad \textrm{and} \quad N^+(v_t)\subseteq (X \setminus \{v_0\}) \cup (Z\setminus \{v_t\}).
	\end{equation}
	
    Choose $i^*\in \{1,\ldots,k-\lceil 3\ell/2 \rceil\}$ such that $v_{i^*}\in N^+(v_t)$. If no such vertex exists, then 
	by~\eqref{eq:neighbourhoods}, we must have
	\[N^+(v_t)\subseteq \{v_{k-\lceil {3\ell}/{2} \rceil+1},\ldots,v_{k-\ell-1}\}\cup (Z\setminus \{v_t\}),\] 
	which implies that 
	\[\deg^+(v_t)\le k-\Big\lfloor \frac{\ell}{2} \Big\rfloor-2< k-\frac{\ell}{2} \leq \delta^0(G), \]
	a contradiction. Note that $v_{i^*}\in X\setminus \{v_0\}$.

	Since $\{v_1,\ldots,v_{k-\lceil 3\ell/2 \rceil}\}\subseteq X$, by Proposition~\ref{prop:special_configurations}~$(ii)$, we may assume that $N^+(v_{i^*-1})\subseteq V(P)$.  
	We now distinguish two cases. 

	\begin{itemize}
		\item Suppose that $N^+(v_{i^*-1})\cap Y=\emptyset$. Clearly, we have 
		\[|N^-(v_0)\cap X|\le k-\ell-1\quad \textrm{and}\quad |N^+(v_{i^*-1})\cap X|\le k-\ell-1.\] 
		Then, since $\delta^0(G)\ge k-{\ell}/{2}$, and $N^-(v_0)\cup N^+(v_{i^*-1})\subseteq X\cup Z$ by~\eqref{eq:neighbourhoods}, it follows that
		\[|N^+(v_{i^*-1})\cap Z|\ge \frac{\ell}{2}+1> \frac{|Z|}{2} \quad \textrm{and}\quad |N^-(v_0)\cap Z|\ge \frac{\ell}{2}+1> \frac{|Z|}{2}. \] 
		Define the set
		\[S:=\{v_j:v_{j-1}\in N^-(v_0)\cap Z\}.\]
        Clearly, $S \subseteq Z$. Moreover, since $v_t\notin N^-(v_0)$ (by~(\ref{eq:neighbourhoods})), we see that $|S|=|N^-(v_0)\cap Z|$. 
		We now have $|S|,|N^+(v_{i^*-1})\cap Z|> {|Z|}/{2}$,  and so by the pigeonhole principle, we must have $S\cap (N^+(v_{i^*-1})\cap Z)\ne \emptyset$. 
		This ensures the existence of a vertex $v_p\in N^+(v_{i^*-1})\cap Z$ such that $v_{p-1}\in N^-(v_0)\cap Z$. Define 
		\[C:=v_0\xrightarrow{P}v_{i^*-1}v_p\xrightarrow{P}v_tv_{i^*}\xrightarrow{P}v_{p-1}v_0,\]
        (see Figure~\ref{fig:hamilton-theorem}). 
        Then $C$ is a directed cycle with $V(C)=V(P)$, i.e.\ $C$ is a Hamilton cycle of $G[V(P)]$. By Proposition~\ref{prop:special_configurations}~$(iii)$, 
		we conclude that $P(\protect\overleftarrow{\ell},\protect\overrightarrow{k-\ell})$ embeds in $G$.

        \begin{figure}[t]
			\centering
			\includegraphics[width=0.4\textwidth]{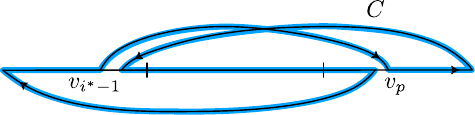}
            \caption{Construction of the Hamilton cycle $C$ in the proof of Theorem~\ref{thm:function} (case $k/2\le \ell \le 2k/3$; subcase $N^+(v_{i^*-1})\cap Y=\emptyset$).}
            \label{fig:hamilton-theorem}
		\end{figure}

		\item Suppose that $N^+(v_{i^*-1})\cap Y\ne \emptyset$, and let $v_p\in  N^+(v_{i^*-1})\cap Y$. 
		\begin{itemize}
			\item Suppose that $p\ge \ell+i^*$. Define the paths
			\[P_1:=v_{i^*-1}v_p \xrightarrow{P} v_{t} \quad \textrm{and} \quad P_2:=v_{i^*-1}\xrightarrow{P} v_{p-1}.\]
            See Figure~\ref{fig:p} (left).  
			Since $\ell+i^*\le p\le t-k+\ell$, we have 
			\[\len(P_1)=t-p+1 > k-\ell \quad \textrm{and} \quad \len(P_2)=p-i^*\ge \ell. \]
			Thus, the path $P(\protect\overleftarrow{\ell},\protect\overrightarrow{k-\ell})$ is a subpath of $\protect\overleftarrow{P_1}P_2$.

			\item Suppose that $p\le \ell+i^*-1$. Choose 
			$j^*\in \{t-k+\ell+1,\ldots,t-\lfloor {\ell}/{2} \rfloor\}$ such that $v_{j^*}\in N^-(v_0)$. The existence of such a vertex $v_{j^*}$
			follows by a similar argument to that of $v_{i^*}$. Indeed, if such a vertex $v_{j^*}$ does not exist, then by~\eqref{eq:neighbourhoods}, we have 
            \[N^-(v_0)\subseteq (X\setminus \{v_0\})\cup \{v_{t-\lfloor \ell/2 \rfloor+1},\dots,v_t\}. \]
            This would imply that \[\deg^-(v_0)\le (k-\ell-1) + \Big\lfloor \frac{\ell}{2} \Big\rfloor = k - \Big\lceil \frac{\ell}{2} \Big\rceil- 1 <\delta^0(G),\]
            a contradiction. 

			Define the paths 
			\[P_1:=v_{j^*}\xrightarrow{P}v_tv_{i^*}\xrightarrow{P} v_{p-1} \quad \textrm{and}\quad 
			P_2:=v_{j^*}v_0\xrightarrow{P} v_{i^*-1}v_{p}\xrightarrow{P} v_{j^*-1}.\] 
            See Figure~\ref{fig:p} (right) for the construction.
            Since $j^*\le t-\lfloor {\ell}/{2} \rfloor$, $i^*\le k-\lceil {3\ell}/{2} \rceil$ and $p\ge k-\ell$ (because $v_p \in Y$), it follows that $t-j^*\ge \lfloor {\ell}/{2} \rfloor$ 
			and $p-i^*\ge \lceil {\ell}/{2} \rceil$. Therefore,
			\[\len(P_1)=t-j^*+p-i^*\ge \ell.\]
			Moreover, since $j^*\ge t-k+\ell+1$, $p\le \ell+i^*-1$ and $t\ge 2\delta^0(G)\ge 2k-\ell$, we obtain
			\[\len(P_2)=i^*+j^*-p\ge i^*+t-k+\ell+1-\ell-i^*+1=t-k+2\ge k-\ell+2.\]
			Thus, the path $P(\protect\overleftarrow{\ell},\protect\overrightarrow{k-\ell})$ is a subpath of $\protect\overleftarrow{P_1}P_2$. 

            \begin{figure}[b]
  \centering
  \hspace{0.05\textwidth}\includegraphics[width=0.4\textwidth]{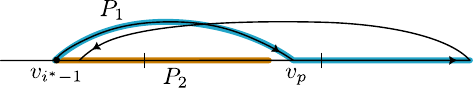}\hfill
  \includegraphics[width=0.4\textwidth]{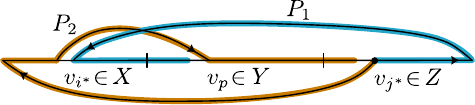}\hspace{0.05\textwidth}
  \caption{Construction of paths $P_1$ and $P_2$ in the proof of Theorem~\ref{thm:function} when $k/2\le \ell \le 2k/3$ and there exists $v_p\in N^+(v_{i^*-1})\cap Y$).}
  \label{fig:p}
\end{figure}
		\end{itemize}

	\end{itemize}
	
	\item Suppose that $\boldsymbol{\ell> 2k/3}$. Here we have $\delta^0(G) \geq 2k/3$, and consequently, $t \geq 2\delta^0(G) \geq 4k/3$.
	Define the sets 
	\[Q:=\{v_{k-\ell},\ldots,v_{t-\ell}\} \quad \textrm{and} \quad R:=\{v_{\ell},\ldots,v_{t-k+\ell}\}.\]
	Note that $|Q| = |R| = t-k+1$. Moreover, note that $Q \cap R \neq \emptyset$ if and only if $\ell \leq t-\ell$, and in this case, we have that $Q \cup R = \{v_{k-\ell},\ldots,v_{t-k+\ell}\}$. 

First, suppose that $N^-(v_0)\cap (Q\cup R)=\emptyset$. 
	Since $N^-(v_0)\subseteq V(P)$ by the maximality of $P$, it follows that $N^-(v_0)\subseteq V(P)\setminus (Q\cup R)$. Then, keeping in mind that $t \geq 2\delta^0(G) \geq 4k/3$ and $\ell > 2k/3$, we obtain 
	\[
	\deg^-(v_0) \le 
	\begin{cases}
		t - 2(t-k+1)  = 2k - t - 2 < 2k/3 & \text{if } Q \cap R = \emptyset, \\
		2(k-\ell)-1 < 2k/3 & \text{if } Q \cap R \neq \emptyset,
	\end{cases}
	\]
	a contradiction to $\delta^0(G) \ge 2k/3$. Thus, we may assume that $N^-(v_0)\cap (Q\cup R)\ne \emptyset$. 
    
		Let $v_i\in N^-(v_0)\cap (Q\cup R)$. Define the paths
		\[P_1:=	v_i v_0 \xrightarrow{P} v_{i-1}  \quad \textrm{and }\quad P_2:=v_i\xrightarrow{P} v_t. \]
		See Figure~\ref{fig:last-case}. 
		We have $\len(P_1)=i$ and $\len(P_2)=t-i$. Then,
		\begin{itemize}
			\item if $v_i\in Q$, then $\len(P_1)\ge k-\ell$ and $\len(P_2)\ge \ell$; and,
			\item if $v_i\in R$, then $\len(P_1)\ge \ell$ and $\len(P_2)\ge k-\ell$. 
		\end{itemize}
		Thus, in both cases, 
		$P(\protect\overleftarrow{\ell},\protect\overrightarrow{k-\ell})$  is a subpath of $\protect\overleftarrow{P_1}P_2$. 
		\begin{figure}[t]
			\centering
			\includegraphics[width=0.85\textwidth]{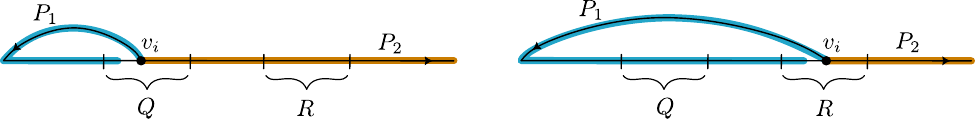}
			\caption{Construction of paths $P_1$ and $P_2$ when $\ell>{2k}/{3}$.}
			\label{fig:last-case}
		\end{figure}
\end{itemize}

\section{Concluding remarks}
As noted in~\cite{klimosova}, the $(k/2)$-blow-up of a directed cycle $C_\ell$ witnesses the tightness of Conjecture~\ref{conj:stein} for the antidirected orientation. In fact, it excludes a broader and natural family of orientations that we call \emph{bouncing paths}. 
For an oriented path $P=p_0,\dots,p_k$, define its \emph{height function} $h:\{0,\dots,k\}\to\mathbb{Z}$ by setting $h_0=0$ and, for $0\le i<k$,
set $h_{i+1}=h_i+1$ when $p_i\to p_{i+1}$, and $h_{i+1}=h_i-1$ when $p_{i+1}\to p_i$. 
We call $P$ \emph{bouncing} if $h_i\in\{-1,0,1\}$ for all $0\le i\le k$; in particular, $h_j=0$ for every even index $j$.
Equivalently, $P$ is bouncing if for every $i$ with $0\le i\le (k/2)-1$, the two consecutive arcs 
$p_{2i}p_{2i+1}$ and $p_{2i+1}p_{2i+2}$ have opposite directions. (In particular, every antidirected path is bouncing.)
For even $k$, every bouncing path $P$ has $k/2+1$ vertices of height $0$ (at positions $0,2,\dots,k$). Any embedding of such a path into the $(k/2)$-blow-up of a directed cycle would force all these vertices to lie in the same vertex class, which is impossible. Hence, the blow-up construction excludes precisely the bouncing paths when $k$ is even.

In a more general setting, Stein's conjecture is near-tight for every orientation: for an even $k$ take a regular tournament $T$ on $k+1$ vertices. We have $\delta^{0}(T) = k/2$ and, by a theorem of Havet and Thomass\'e~\cite{havetthomasse2000} (proved first for large $k$ by Thomason~\cite{thomason1986paths}), $T$ contains every oriented path with $k$ arcs (with three small exceptions all involving the antidirected path), but clearly $T$ does not contain any path with $k+1$ arcs. 
Thus we wonder if Conjecture~\ref{conj:stein} can be strengthened for all orientations except the bouncing paths of even length $k$.

\begin{question}
\label{question:unique}
Is it true that every oriented graph $G$ with $\delta^0(G)\ge k/2$ contains every orientation of the $k$-edge path, except for the bouncing paths when $k$ is even?
\end{question}

\noindent
Note that Jackson's theorem (Theorem~\ref{thm:jackson}) gives an affirmative answer for the directed orientation.








\bibliographystyle{abbrv}
\bibliography{biblio} 

\end{document}